\documentclass[11pt,leqno]{article}
\usepackage[margin=1in]{geometry} 
\usepackage{amssymb,amsfonts,amsmath,bbm,mathrsfs,stmaryrd,mathtools}
\usepackage{xcolor}
\usepackage{url}

\usepackage{graphicx}

\usepackage{accents}

\usepackage{extarrows}

\usepackage[shortlabels]{enumitem}
\usepackage{tensor}

\usepackage{xr}
\usepackage[T1]{fontenc}
\usepackage[utf8]{inputenc}
\usepackage[colorlinks,
linkcolor=black!75!red,
citecolor=blue,
pdftitle={},
pdfproducer={pdfLaTeX},
pdfpagemode=None,
bookmarksopen=true,
bookmarksnumbered=true,
backref=page]{hyperref}

\usepackage{tikz}
\usetikzlibrary{arrows,calc,decorations.pathreplacing,decorations.markings,decorations.shapes,intersections,shapes.geometric,through,fit,shapes.symbols,positioning,decorations.pathmorphing}

\makeatletter
\newlength\zig@L
\newlength\zig@La
\newlength\zig@Lb

\newcommand{\xzigrightarrow}[2][]{%
  \mathrel{%
    \settowidth{\zig@La}{$\scriptstyle #2$}%
    \settowidth{\zig@Lb}{$\scriptstyle #1$}%
    \zig@L=\zig@La\relax
    \ifdim\zig@Lb>\zig@L \zig@L=\zig@Lb\fi
    \advance\zig@L by 2.2em\relax
    \tikz[baseline=-0.65ex]{%
      \draw[->,
            line cap=round,
            decorate,
            decoration={zigzag,segment length=4pt,amplitude=1.1pt}]%
        (0,0) -- (\zig@L,0)
        node[midway,above=2pt] {$\scriptstyle #2$}%
        \if\relax\detokenize{#1}\relax\else
          node[midway,below=2pt] {$\scriptstyle #1$}%
        \fi
      ;
    }%
  }%
}
\makeatother

\makeatletter
\newcommand{\squigjoin}{1mu} % tune this: 0.5mu, 1mu, 1.5mu, ...

\def\sqleft@{\sim}                    % no overlap here
\def\sqmid@{\sim\mkern-\squigjoin}    % overlap only between repeated mids

\def\rightsquigarrowfill@{%
  \arrowfill@{\sqleft@}{\sqmid@}{\mkern-4mu\succ}%
}

\newcommand{\xrightsquigarrow}[2][]{%
  \ext@arrow 0359\rightsquigarrowfill@{#1}{#2}%
}
\makeatother

\makeatletter
\newcommand*\circled[1]{\tikz[baseline=(char.base)]{
    \node[shape=circle, draw, inner sep=0pt, 
    minimum height={\f@size},] (char) {\vphantom{WAH1g}#1};}}
\makeatother

\makeatletter
\DeclareRobustCommand\widecheck[1]{{\mathpalette\@widecheck{#1}}}
\def\@widecheck#1#2{%
    \setbox\z@\hbox{\m@th$#1#2$}%
    \setbox\tw@\hbox{\m@th$#1%
       \widehat{%
          \vrule\@width\z@\@height\ht\z@
          \vrule\@height\z@\@width\wd\z@}$}%
    \dp\tw@-\ht\z@
    \@tempdima\ht\z@ \advance\@tempdima2\ht\tw@ \divide\@tempdima\thr@@
    \setbox\tw@\hbox{%
       \raise\@tempdima\hbox{\scalebox{1}[-1]{\lower\@tempdima\box
\tw@}}}%
    {\ooalign{\box\tw@ \cr \box\z@}}}
\makeatother

\usepackage{braket}

\usepackage[amsmath,thmmarks,hyperref]{ntheorem}
\usepackage{cleveref}

\newcommand\nthalias[1]{\AddToHook{env/#1/begin}{\crefalias{lemma}{#1}}}

\nthalias{assumption}
\nthalias{conjecture}
\nthalias{construction}
\nthalias{convention}
\nthalias{corollaryN}
\nthalias{corollary}
\nthalias{definition}
\nthalias{examples}
\nthalias{example}
\nthalias{lemma}
\nthalias{notation}
\nthalias{propositionN}
\nthalias{proposition}
\nthalias{question}
\nthalias{recollection}
\nthalias{remarks}
\nthalias{remark}
\nthalias{sketch}
\nthalias{theoremN}
\nthalias{theorem}

\creflabelformat{enumi}{#2#1#3}

\crefname{section}{Section}{Sections}
\crefformat{section}{#2Section~#1#3} 
\Crefformat{section}{#2Section~#1#3} 

\crefname{subsection}{\S}{\S\S}
\AtBeginDocument{%
  \crefformat{subsection}{#2\S#1#3}%
  \Crefformat{subsection}{#2\S#1#3}% 
}

\crefname{subsubsection}{\S}{\S\S}
\AtBeginDocument{%
  \crefformat{subsubsection}{#2\S#1#3}%
  \Crefformat{subsubsection}{#2\S#1#3}% 
}

\theoremstyle{plain}

\newtheorem{lemma}{Lemma}[section]
\newtheorem{proposition}[lemma]{Proposition}
\newtheorem{corollary}[lemma]{Corollary}
\newtheorem{theorem}[lemma]{Theorem}

\theoremstyle{plain}
\theoremnumbering{Alph}
\newtheorem{theoremN}{Theorem}

\theoremstyle{plain}
\theorembodyfont{\upshape}
\theoremsymbol{\ensuremath{\blacklozenge}}

\newtheorem{definition}[lemma]{Definition}
\newtheorem{example}[lemma]{Example}

\newtheorem{remark}[lemma]{Remark}
\newtheorem{remarks}[lemma]{Remarks}

\crefname{definition}{definition}{definitions}
\crefformat{definition}{#2definition~#1#3} 
\Crefformat{definition}{#2Definition~#1#3} 

\crefname{ex}{example}{examples}
\crefformat{example}{#2example~#1#3} 
\Crefformat{example}{#2Example~#1#3} 

\crefname{exs}{example}{examples}
\crefformat{examples}{#2example~#1#3} 
\Crefformat{examples}{#2Example~#1#3} 

\crefname{remark}{remark}{remarks}
\crefformat{remark}{#2remark~#1#3} 
\Crefformat{remark}{#2Remark~#1#3} 

\crefname{remarks}{remark}{remarks}
\crefformat{remarks}{#2remark~#1#3} 
\Crefformat{remarks}{#2Remark~#1#3} 

\crefname{convention}{convention}{conventions}
\crefformat{convention}{#2convention~#1#3} 
\Crefformat{convention}{#2Convention~#1#3} 

\crefname{notation}{notation}{notations}
\crefformat{notation}{#2notation~#1#3} 
\Crefformat{notation}{#2Notation~#1#3} 

\crefname{table}{table}{tables}
\crefformat{table}{#2table~#1#3} 
\Crefformat{table}{#2Table~#1#3}

\crefname{lemma}{lemma}{lemmas}
\crefformat{lemma}{#2lemma~#1#3} 
\Crefformat{lemma}{#2Lemma~#1#3} 

\crefname{proposition}{proposition}{propositions}
\crefformat{proposition}{#2proposition~#1#3} 
\Crefformat{proposition}{#2Proposition~#1#3} 

\crefname{propositionN}{proposition}{propositions}
\crefformat{propositionN}{#2proposition~#1#3} 
\Crefformat{propositionN}{#2Proposition~#1#3} 

\crefname{corollary}{corollary}{corollaries}
\crefformat{corollary}{#2corollary~#1#3} 
\Crefformat{corollary}{#2Corollary~#1#3} 

\crefname{corollaryN}{corollary}{corollaries}
\crefformat{corollaryN}{#2corollary~#1#3} 
\Crefformat{corollaryN}{#2Corollary~#1#3} 

\crefname{theorem}{theorem}{theorems}
\crefformat{theorem}{#2theorem~#1#3} 
\Crefformat{theorem}{#2Theorem~#1#3} 

\crefname{theoremN}{theorem}{theorems}
\crefformat{theoremN}{#2theorem~#1#3} 
\Crefformat{theoremN}{#2Theorem~#1#3} 

\crefname{enumi}{}{}
\crefformat{enumi}{#2#1#3}
\Crefformat{enumi}{#2#1#3}

\crefname{assumption}{assumption}{Assumptions}
\crefformat{assumption}{#2assumption~#1#3} 
\Crefformat{assumption}{#2Assumption~#1#3} 

\crefname{construction}{construction}{Constructions}
\crefformat{construction}{#2construction~#1#3} 
\Crefformat{construction}{#2Construction~#1#3} 

\crefname{sketch}{sketch}{Sketches}
\crefformat{sketch}{#2sketch~#1#3} 
\Crefformat{sketch}{#2Sketch~#1#3} 

\crefname{recollection}{recollection}{Recollections}
\crefformat{recollection}{#2recollection~#1#3} 
\Crefformat{recollection}{#2Recollection~#1#3} 

\crefname{question}{question}{Questions}
\crefformat{question}{#2question~#1#3} 
\Crefformat{question}{#2Question~#1#3} 

\crefname{equation}{}{}
\crefformat{equation}{(#2#1#3)} 
\Crefformat{equation}{(#2#1#3)}

\numberwithin{equation}{section}

\theoremstyle{nonumberplain}
\theoremsymbol{\ensuremath{\blacksquare}}

\newtheorem{proof}{Proof}
\newcommand\pf[1]{\newtheorem{#1}{Proof of \Cref{#1}}}

\newcommand\bC{{\mathbb C}}
\newcommand\bD{{\mathbb D}}

\newcommand\bG{{\mathbb G}}

\newcommand\bK{{\mathbb K}}

\newcommand\bR{{\mathbb R}}
\newcommand\bS{{\mathbb S}}

\newcommand\bU{{\mathbb U}}

\newcommand\bZ{{\mathbb Z}}

\newcommand\cA{{\mathcal A}}
\newcommand\cB{{\mathcal B}}
\newcommand\cC{{\mathcal C}}

\newcommand\cE{{\mathcal E}}
\newcommand\cF{{\mathcal F}}
\newcommand\cG{{\mathcal G}}
\newcommand\cH{{\mathcal H}}

\newcommand\cL{{\mathcal L}}

\newcommand\cO{{\mathcal O}}

\newcommand\cU{{\mathcal U}}
\newcommand\cV{{\mathcal V}}

\newcommand\ol{\overline}

\DeclareMathOperator{\id}{id}

\DeclareMathOperator{\End}{\mathrm{End}}

\DeclareMathOperator{\spn}{\mathrm{spn}}

\newcommand{\comment}[1]{}

\newcommand{\xrightarrowdbl}[2][]{%
  \xrightarrow[#1]{#2}\mathrel{\mkern-14mu}\rightarrow
}

\title{Obstructed subhomogeneous-bundle extensions and embeddings}
\author{Alexandru Chirvasitu}

\begin{document}

\date{}

\newcommand{\Addresses}{{% additional braces for segregating \footnotesize
  \bigskip
  \footnotesize

  \textsc{Department of Mathematics, University at Buffalo}
  \par\nopagebreak
  \textsc{Buffalo, NY 14260-2900, USA}  
  \par\nopagebreak
  \textit{E-mail address}: \texttt{achirvas@buffalo.edu}

}}

\maketitle

\begin{abstract}
  We address a number of problems concerning the (im)possibility of either extending locally trivial subbundles of possibly singular Banach/$C^*$ bundles globally, embedding subhomogeneous bundles into homogeneous ones, or recovering locally trivial compact-Lie-group-equivariant Banach or $C^*$ bundles as pullbacks along equivariant maps to compact spaces. The results include (1) the global extensibility of a locally trivial Banach/Hilbert/Banach-algebra/$C^*$ subbundle from a closed subspace of a paracompact space given appropriate homotopy constraints; (2) the homogeneous embeddability of equivariant subhomogeneous Banach/Hilbert bundles locally trivial along the singular locus under the same homotopy constraints, and (3) the characterization of finite-type equivariant locally trivial subhomogeneous $C^*$ bundles on normal spaces as precisely those (a) locally trivial as plain vector bundles, or (b) pulled back from the universal equivariant compactification or (c) pulled back from an equivariant map into a smooth manifold. The latter extends results of Phillips concerning non-equivariant matrix-algebra bundles restricted along the Stone\v{C}ech compactification. 
\end{abstract}


\vspace{.5cm}

\noindent{MSC 2020: 46M20; 55R10; 55R15; 54D20; 55R91; 46L30; 46H25; 46L85
  % 46M20   	Methods of algebraic topology in functional analysis (cohomology, sheaf and bundle theory, etc.)
  % 55R10   	Fiber bundles in algebraic topology
  % 55R15   	Classification of fiber spaces or bundles in algebraic topology
  % 54D20   	Noncompact covering properties (paracompact, Lindelöf, etc.)
  % 55R91   	Equivariant fiber spaces and bundles in algebraic topology
  % 46L30   	States of selfadjoint operator algebras
  % 46H25   	Normed modules and Banach modules, topological modules (if not placed in 13-XX or 16-XX)
  % 46L85   	Noncommutative topology  
}

%\tableofcontents

%%%%%%%%%%%%%%%%%%%%%%%%%%%%%%%%
%%%%%%%%%%%%%%%%%%%%%%%%%%%%%%%%
\section*{Introduction}

The questions forming the focus of the paper are offshoots of the theory of \emph{non-commutative} (or \emph{quantum}) \emph {branched covers}, introduced in \cite[Definition 1.2]{pt_brnch} by analogy to the purely classical version: continuous open surjections of compact $T_2$ spaces with global upper bounds on fiber sizes. Quantum branching is also studied in \cite{bg_cx-exp} where, due to the ``half-classical'' nature of the work, the link to the theory of (singular) \emph{$C^*$-algebra bundles} \cite[p.9]{dg_ban-bdl} is in evidence: for purposes as encompassing as the present paper need accommodate, \cite[\S 1]{bg_cx-exp}'s quantum branched covers are subhomogeneous $C^*$ bundles $\cA\xrightarrowdbl{}X$ over compact Hausdorff spaces which admit \emph{finite-index expectations} \cite[Definition 2]{fk_fin-ind} $\Gamma(\cA)\xrightarrowdbl{}C(X)$. 

An examination of the cited sources and others touching on the topic (\cite{zbMATH05172034} for instance) and their proof techniques reveals intimate connections between the attendant problems (e.g. existence of the requisite finite-index expectations under various sets of assumptions) and matters involving
\begin{itemize}[wide]
\item the embeddability of possibly singular subhomogeneous $C^*$/Banach bundles into locally trivial counterparts;
\item variant questions seeking necessary/sufficient conditions ensuring the embeddability of a locally trivial bundle into a globally trivial one;
\item and the possibility of extending locally trivial $C^*$/Banach subbundles from exceptional loci of the base space either locally or globally. 
\end{itemize}
It is with these types of questions that the paper is concerned: the extent to which such extensions/embeddings exist, the analysis of possible obstructions when they do not, and the utility of packaging those obstructions as invariants attached to possibly singular Banach or $C^*$ bundles. The objects involved will moreover frequently be \emph{$\bU$-equivariant} \cite[\S I.8]{td_transf-gp} for compact Lie groups $\bU$ operating compatibly on everything in sight (\emph{bundles$_{\bU}$} is the term employed throughout the paper).

To motivate a first sampling of the results, recall \cite[Proposition 2.9]{zbMATH05172034}'s characterization of \emph{finite-type} matrix-algebra bundles (i.e. those trivializable over the members of a finite cover) on normal spaces as precisely those induced from some compactification of the base space. The problem fits the broad context outlined above, and can be regarded either as a matter of bundle extensibility across a compactification or as an examination of obstructions to various flavors of trivial embeddability. \Cref{th:ft.emb.ft,th:ft.ext.cpct} extend the result just cited by superimposing the equivariance layer and addressing arbitrary subhomogeneous locally trivial $C^*$ bundles; a compressed summary reads as follows:

\begin{theoremN}\label{thn:ft.intro}
  Let $\bU$ be a compact Lie group. The following conditions on a locally trivial subhomogeneous $C^*$ bundle$_{\bU}$ $\cA\xrightarrowdbl{}X$ over a normal $\bU$-space are equivalent. 

  \begin{enumerate}[(a),wide]
  \item $\cA$ is finite-type.

  \item $\cA$ is finite-type as a vector bundle. 

  \item $\cA$ extends across the \emph{universal $\bU$-equivariant compactification} \cite[\S 1]{megr_max-equiv-cpct} $X\lhook\joinrel\to \beta_{\bU}X$.

  \item $\cA$ is a pullback through an equivariant map to a smooth compact $\bU$-manifold. 
  \end{enumerate}
\end{theoremN}

As to extension from singular loci, the rigidity characteristic of local triviality already affords \emph{local} extensibility of locally trivial subhomogeneous subbundles around closed subsets of paracompact spaces: \cite[Theorems 0.2 and 0.3]{MR5060733}. Further homotopy constraints ensure that goes through globally:

\begin{theorem}\label{th:ext.loc.triv.hmtp.eq}
  Let $\bU$ be a compact Lie group, and
  \begin{itemize}[wide]
  \item $F\subseteq X$ a $\bU$-equivariant closed embedding of paracompact $\bU$-spaces with paracompact $X\setminus F$;

  \item such that $U$ and $X$ are $\bU$-equivariantly homotopic \emph{relative to $F$} \cite[p.3]{hatch_at} for arbitrarily small $\bU$-invariant neighborhoods $U\supseteq F$ with paracompact $U\setminus F$.
  \end{itemize}

  \begin{enumerate}[(1),wide]
  \item\label{item:th:ext.loc.triv.hmtp.eq:ban} If $\cE\xrightarrowdbl{\pi} X$ is a continuous Banach bundle$_{\bU}$, locally trivial subhomogeneous on $X\setminus F$, a locally trivial subhomogeneous subbundle$_{\bU}$ $\cF\le \cE|_{F}$ extends globally and $\bU$-equivariantly to
    \begin{equation*}
      \left(\ol{\cF}\xrightarrowdbl{\quad}X\right)
      \le
      \left(\cE\xrightarrowdbl{\quad}X\right)
      ,\quad
      \ol{\cF}|_{F}=\cF. 
    \end{equation*}

  \item\label{item:th:ext.loc.triv.hmtp.eq:hilb} The analogue of \Cref{item:th:ext.loc.triv.hmtp.eq:ban} holds for Hilbert-bundles$_{\bU}$ $\cF\le \cE|_F$.

  \item\label{item:th:ext.loc.triv.hmtp.eq:balg} If $\cA\xrightarrowdbl{\pi}X$ is a unital Banach-algebra bundle$_{\bU}$, locally trivial, subhomogeneous semisimple on $X\setminus F$ with locally trivial subhomogeneous semisimple $\cB\le \cA|_F$ there is a global equivariant extension
    \begin{equation*}
      \left(\ol{\cB}\xrightarrowdbl{\quad}X\right)
      \le
      \left(\cA\xrightarrowdbl{\quad}X\right)
      ,\quad
      \ol{\cB}|_{F}=\cB. 
    \end{equation*}

  \item\label{item:th:ext.loc.triv.hmtp.eq:cast} The analogue of \Cref{item:th:ext.loc.triv.hmtp.eq:balg} holds for $C^*$ bundles$_{\bU}$. 
  \end{enumerate}
\end{theorem}

This in turn has consequences on singular-bundle embeddability into locally trivial bundles. 

\begin{theorem}\label{th:can.emb.ban.hilb}
  Let $\bU$ be a compact Lie group, and
  \begin{itemize}[wide]
  \item $F\subseteq X$ a $\bU$-equivariant closed embedding of paracompact $\bU$-spaces with paracompact $X\setminus F$;

  \item such that $U$ and $X$ are $\bU$-equivariantly homotopic \emph{relative to $F$} \cite[p.3]{hatch_at} for arbitrarily small $\bU$-invariant neighborhoods $U\supseteq F$ with paracompact $U\setminus F$.
  \end{itemize}
  \begin{enumerate}[(1),wide]
  \item A continuous subhomogeneous Banach bundle$_{\bU}$ $\cE\xrightarrowdbl{}X$ locally trivial over $F$ and $X\setminus F$ and finite-type over the latter embeds into a trivial homogeneous Banach bundle. 

  \item The Hilbert-bundle analogue also holds. 
  \end{enumerate}
\end{theorem}

In much the same spirit, \Cref{th:cn.ext} packages the non-extensibility of a Banach bundle
\begin{itemize}[wide]
\item rank-$r_0$ trivial on the \emph{punctured cone}
  \begin{equation*}
    \cC_{\times}X
    :=
    \cC X\setminus \left(0_X:=\pi^{-1}\left(X\times \{0\}\right)\right)
    ,\quad
    X\times [0,1]
    \xrightarrowdbl{\quad \pi\quad}
    \cC X:=X\times [0,1]/\left(X\times\{0\}\right)
  \end{equation*}
  on a paracompact space;

\item and rank-$r_1$ at the tip $0_X$ of the cone
\end{itemize}
as a homotopy-theoretic obstruction class. Non-extensibility for appropriate numerical data $0<r_1< r_0$ is then illustrated by \Cref{ex:not.emb.twisted.emb} via \Cref{cor:non.ext.spec.chi}.

% %

%%%%%%%%%%%%%%%%%%%%%%%%%%%%%%%%
%%%%%%%%%%%%%%%%%%%%%%%%%%%%%%%%
\section{Extensions and embeddings of (equivariant) bundles}\label{se:ext.emb}

The \emph{Banach bundles} featuring throughout are those of \cite[Definition II.13.4]{fd_bdl-1}, \cite[Definition 1.1]{dg_ban-bdl}, \cite[\S 10, Definition, p.100]{fell_ind-ban-bdl}, etc.:

\begin{itemize}[wide]
\item continuous open surjections $\cE\xrightarrow{\pi}X$ with Banach-space \emph{fibers} $(\cE_x,\|\cdot\|_x)$ for $\cE_x:=\pi^{-1}(x)$;
\item with addition and scaling continuous in the obvious sense;
\item for every $x\in X$ the sets
  \begin{equation*}
    \left\{v\in\pi^{-1}U\ :\ \|v\|_y<\varepsilon,\ y\in U\ni x\right\}
    ,\quad
    U=\accentset{\large\circ}{U}\ni x
    ,\quad
    \varepsilon>0
  \end{equation*}
  constitute a local neighborhood system around $0_x\in \cE_x$;

\item and the norm function $\cE\ni v\mapsto \|v\|_{\pi v}\in \bR$ is continuous the vast majority of times (for the \emph{continuous} or \emph{(F)} Banach bundles of \cite[p.8]{dg_ban-bdl}) but occasionally perhaps only upper semicontinuous (\emph{(H) Banach bundles}).
\end{itemize}
Hilbert, Banach-algebra or $C^*$ bundles are defined analogously \cite[p.9]{dg_ban-bdl}, fibers being equipped with the respective richer structures. More generally, we borrow the language of \cite[Definition 0.1]{MR5060733} in speaking of \emph{$\bU$-equivariant} Banach/Hilbert/Banach-algebra/$C^*$ bundles, or simply Banach/etc. \emph{ bundles$_{\bU}$} for a topological group (mostly compact Lie) operating in a structure-preserving fashion on everything of concern: $\cE$ as well as $X$, with $\pi$ and the relevant structure maps $\bU$-equivariant.

Apart from the Banach-flavored notion just recalled the phrase ``bundle'' applies below (along with its $\bU$-equivariant extensions) also to other familiar contexts: principal bundles$_{\bU}$ (as discussed in \cite[\S I.8]{td_alg-top} where the acting group $\bU$ is denoted by $\Gamma$), equivariant vector bundles, etc. The notion of equivariance employed is thus intermediate in strength between \cite{MR650393} (and the more recent \cite{MR3331607}) say, where $\bU$ acts trivially on the structure group of the bundle, and \cite{MR885537}, where it is packaged rather as an extension of $\bU$ by $\bG$ (which would specialize to the semidirect product $\bG\rtimes \bU$ in the present context). 

Base spaces of bundles will always be assumed at least \emph{Tychonoff} (i.e. \cite[Definition 14.8]{wil_top} satisfying the $T_{3\frac{1}{2}}$ separation axiom), and usually \emph{paracompact} \cite[Definition 20.6]{wil_top}. Banach/Hilbert/Banach-algebra/$C^*$ bundles are
\begin{itemize}[wide]
\item \emph{(sub)homogeneous} (e.g. in concert with \cite[Definitions 2.4 and 2.5]{bg_cx-exp}, \cite[\S 2]{gog_top-fg}) are those for which the fibers are mutually isomorphic finite-dimensional objects (or have dimensions bounded by a global finite constant respectively);
\item similarly, \emph{$n$-(sub)homogeneous} if the fiber dimension is precisely (respectively upper-bounded by) $n\in \bZ_{\ge 0}$;

\item and \emph{locally trivial} \cite[Remark II.13.9]{fd_bdl-1} if identifiable with $U\times E$ for an object $E$ of the appropriate type (Banach space, etc.) for open subsets $U\subseteq X$ covering $X$ (the phrase conforms to standard language in bundle theory generally principal, vector, etc.: \cite[\S 3.1]{td_alg-top}, say). 
\end{itemize}

It will be worth noting how some of these notions coalesce under appropriate topological conditions.

\begin{proposition}\label{pr:loc.homog.loc.triv}
  Let $\bU$ be a compact Lie group and $X$ a paracompact $\bU$-space. A continuous subhomogeneous Banach/Hilbert/Banach-algebra/$C^*$ bundle$_{\bU}$ $\cE\xrightarrowdbl{}X$ with semisimple fibers in the Banach-algebra case is locally homogeneous if and only if it is locally trivial. 
\end{proposition}
\begin{proof}
  The backward implication is self-evident, so it is $(\Rightarrow)$ that forms the substance of the claim. 

  Given local triviality over the compact orbits of the action $\bU\circlearrowright X$, its local equivariant extensibility around such an orbit follows from \cite[Theorems 0.2 and 0.3]{MR5060733}. As to the orbits $Y\subseteq X$, the local triviality of $\cE|_Y$ both
  \begin{itemize}[wide]
  \item follows from the same proof technique (employing Johnson's \cite[Theorem 3.1]{john_approx} to the effect that almost-multiplicative morphisms of Banach/$C^*$-algebras with semisimple finite-dimensional domain are metrically close to morphisms);

  \item or can be extracted from literature precursors handling locally compact base spaces: directly in the $C^*$ case from \cite[Theorem 3.1]{fell_struct} (which \cite[Theorems 0.3]{MR5060733} generalizes) and by analogy in the other cases. 
  \end{itemize}
\end{proof}

\begin{remark}\label{re:loc.triv.clrf}
  Note that the local identifications with trivial objects relevant to \Cref{pr:loc.homog.loc.triv} are \emph{topological} but not \emph{metric}: the operative notion of isomorphism $\left(\cE\xrightarrowdbl{} X\right) \cong \left(\cF\xrightarrowdbl{} X\right)$ entails linear and topological (also multiplicative in the algebra case) fiber maps $\cE_x\cong \cF_x$ which need not be isometric. Or: what the cited source \cite[Definition II.13.8]{fd_bdl-1} refers to as \emph{isomorphisms} while \cite[p.23]{dg_ban-bdl} terms \emph{weak isomorphisms} (reserving the stronger version for norm-preserving variants).

  That the metric analogue of \Cref{pr:loc.homog.loc.triv} cannot hold is self-evident: a Banach bundle might be Hilbert along exactly one fiber $\cE_x$ but not along $\cE_y$ for $y$ clustering towards $x$. The universal example of this nature would have
  \begin{itemize}[wide]
  \item the space $X$ of compact convex origin-symmetric ($n$-dimensional) convex bodies in $\bR^n$ for $n\ge 2$ as its base, with the \emph{Hausdorff-metric}-induced \cite[\S 7.3.1]{bbi} topology;
  \item with the fiber at $x\in X$ the unique Banach space $\bR^n$ having the convex body $x$ as its unit ball (per the usual \cite[post Lemma 1.7.12]{schn_cvx_2e_2014} convex body/Banach structure correspondence). 
  \end{itemize}
\end{remark}

Recall \cite[Definition 3.5.7]{hus_fib} that a \emph{finite-type} bundle $\cE\xrightarrowdbl{\hspace{0pt}}X$ is one trivializable over (each member of) a finite open cover of $X$. The notion also makes sense in the context of $\bU$-equivariant vector bundles, local triviality, finite type and \emph{numerability} being as in \cite[Definition I.9.1]{td_transf-gp} (in conjunction with \cite[\S I.6]{td_transf-gp} for broader notions of equivariant numerability). Following \cite[\S\S 2.1 and 2.2]{td_alg-top}, $[X,Y]$ and $[X,Y]^0$ denote classes of (basepoint-preserving) maps $X\to Y$ up to homotopy (respectively base-point-preserving homotopy). 

There are ways in which homogeneous extensibility might be topologically obstructed, unrelated to issues of finite type.

\begin{example}\label{ex:not.emb.twisted.emb}
  All bundles in the present discussion are real, for some variety (analogous complex examples can be constructed). Consider
  \begin{itemize}[wide]
  \item a trivial rank-$r_0$ Hilbert bundle over a punctured ball
    \begin{equation}\label{eq:h0dd}
      \cH^0
      \xrightarrowdbl{\quad}
      \left(
        \bD^{d}_{\times}
        :=
        \bD^{d}\setminus\{0\}
      \right)
      ,\quad
      \bD^{d}:=\left\{x\in \bR^{d}\ |\ \|x\|\le 1\right\};
    \end{equation}
    
  \item a rank-$(r_1\le r_0)$ Hilbert bundle $\cH^{1}$ over the smaller ball $\frac 12 \bD^{d}$;
    
  \item and the resulting (F) Hilbert bundle $\cH\xrightarrowdbl{\hspace{0pt}}\bD^d$ formed by gluing the two via a continuous map
    \begin{equation}\label{eq:char.fn.2.isoms}
      \frac 12 \bD^d_{\times}
      =
      \left(\bD^d_{\times}\cap \frac 12 \bD^d\right)
      \xrightarrow{\quad \chi\quad}
      \left(
        \mathrm{Iso}_{r_1\to r_0}
        :=
        \left\{\text{isometries }\bR^{r_1}\xrightarrow{\quad} \bR^{r_0}\right\}
      \right)
    \end{equation}
    (a generalization of the familiar \emph{clutching} construction: \cite[Definition 2.7.2]{hjjm_bdle}, \cite[Proposition 10.7.1]{hus_fib}, etc.).
  \end{itemize}
  Formally, $\cH$ can be constructed via \cite[Proposition 1.3]{dg_ban-bdl} as follows.
  \begin{itemize}[wide]
  \item Consider, in first instance, a \emph{Hilbert family} $\cH\xrightarrowdbl{\pi}\bD^d$ (i.e. \cite[p.8]{dg_ban-bdl} a surjective map with Hilbert-space fibers $\cH_z:=\pi^{-1}z$) with
    \begin{equation*}
      \cH|_{\bD^d_{\times}}=\cH^{0}
      \quad\text{and}\quad
      \cH_0=\cH^1_0.
    \end{equation*}
    
  \item The subspace
    \begin{equation*}
      \Gamma_{00}(\cH)
      :=
      \left\{
        \bD^d\xrightarrow{s}\cH
        \ :\
        \begin{aligned}
                    \pi\circ s
                    &=\id_{\bD^d}\\
          s|_{\bD^d_{\times}}
          &\in \Gamma_{b}\left(\cH^0\right)\\
          s|_{F}
          &\in \Gamma\left(\cH^1\right)
            \ \text{for some closed neighborhood $F\ni 0$}
        \end{aligned}
      \right\},
    \end{equation*}
    with the last two conditions compatible upon embedding
    \begin{equation*}
      \cH^1|_{\frac 12 \bD^d_{\times}}
      \lhook\joinrel\xrightarrow{\quad \text{via $\chi$}\quad}
      \cH^0|_{\frac 12 \bD^d_{\times}},
    \end{equation*}
    is \emph{total} for $\cH$ in the sense \cite[p.14]{dg_ban-bdl} that
    \begin{equation*}
      \forall(x\in \bD^d)
      \quad:\quad
      \overline{\spn\left\{s(x)\ :\ s\in \Gamma_{00}\right\}}
      =
      \cH_x
    \end{equation*}
    with continuous norm, thus \cite[Proposition 1.3]{dg_ban-bdl} affording $\pi$ a unique Hilbert-bundle structure with $\Gamma_{00}\le \Gamma(\pi)$. 
  \end{itemize}
  $\mathrm{Iso}_{r_1\to r_0}$ being what \cite[\S 7.7]{steen_fib} refers to as the \emph{Stiefel manifold} $V_{r_0,r_1}$, we have \cite[Theorem 25.6]{steen_fib}
  \begin{equation}\label{eq:pi.iso.stief}
    \left[\frac 12\bD^d_{\times},\mathrm{Iso}_{r_1\to r_0}\right]
    \cong
    \pi_{d-1}\left(\mathrm{Iso}_{r_1\to r_0}\right)    
    \cong
    \begin{cases}
      \bZ&\text{if }d-1=r_0-r_1\text{ is even or }r_1=1\\
      \bZ/2&\text{if }d-1=r_0-r_1\text{ is odd and }r_1>1.
    \end{cases}
  \end{equation}
  Homotopically non-trivial $\chi$ are thus available whenever $d-1=r_0-r_1>0$ (and $r_1>0$). Selecting $\chi$ and the numerical parameters judiciously, \Cref{cor:non.ext.spec.chi} below confirms that $\cH$ will not extend across the singularity at $0\in \bD^d$ to a locally trivial rank-$r_0$ bundle, when, say
  \begin{itemize}
  \item $d-1=r_0-r_1=2$, for then indeed $\pi_{d-1}=\pi_2$ vanishes for $SO(r_0)$ \cite[\S 22.10]{steen_fib} (as in fact it does for all Lie groups whatsoever: \cite[Proposition V.7.5]{btd_lie_1995} for the compact case; hence also generally for arbitrary connected Lie groups, homotopy-equivalent to compact ones \cite[Theorem 13]{iw}).

  \item and $[\chi]$ is non-trivial, which according to \Cref{eq:pi.iso.stief} is achievable for $r_1>0$ and $r_0-r_1>0$ whenever $d-1=r_0-r_1$.
  \end{itemize}
\end{example}

\Cref{th:cn.ext} distills some of the general principles at work in \Cref{ex:not.emb.twisted.emb}, with the following useful piece of terminology formalizing the ad-hoc construction employed there; cf. \cite[\S 17.12]{steen_fib} for a notion of bundle \emph{characteristic map/class} very much in the same spirit. Throughout the sequel,
\begin{equation*}
  X\times \left(I:=[0,1]\right)
  \xrightarrowdbl{\quad}
  \cC X
  :=
  X\times [0,1]/X\times\{0\}
\end{equation*}
denotes the \emph{cone} \cite[Example 9.12(f)]{wil_top} on a space $X$ with variants such as $\cC_{\le \varepsilon}X$ or $\cC_{<\varepsilon}X$ denoting the images of those subsets of $X\times I$ where the second coordinate satisfies the condition indicated in the subscript (i.e. is $\le \varepsilon$ or $<\varepsilon$, in the two instances). We write $0_X\in \cC X$ for the tip $\cC_{\{0\}}X$. 

\begin{definition}\label{def:char.cl}
  Let $X$ be a paracompact topological space and $\cE\xrightarrowdbl{\hspace{0pt}}\cC X$ a continuous Banach $\Bbbk$-bundle, $\Bbbk\in \left\{\bR,\bC\right\}$
  \begin{itemize}[wide]
  \item trivial of rank $r_0$ over the punctured cone $\cC_{\times}X:=\cC X\setminus 0_X$;

  \item and of rank $r_1$ at $0_X$ (so in particular $r_0\ge r_1$, by continuity).  
  \end{itemize}
  \begin{enumerate}[(1),wide]
  \item\label{item:def:char.cl:[]} Extend a frame (ordered basis) of the fiber $\cE_{0_X}$ locally around $0_X$, as always possible \cite[post Theorem 5.1]{hk_shv-bdl}. In a sufficiently small neighborhood $\cC_{\le \varepsilon}X\ni 0_X$ the frame will remain linearly-independent, whence a map
    \begin{equation}\label{eq:chi.on.cex}
      \cC_{\le \varepsilon,\times}X
      \xrightarrow{\quad\chi\quad}
      \mathrm{Emb}(r_1\to r_0)
      :=
      \left\{\text{linear injections }\Bbbk^{r_1}\le \Bbbk^{r_0}\right\}
      \quad
      \left(\text{some $\varepsilon>0$}\right).
    \end{equation}
    The \emph{characteristic map (or class)} $[\chi]$ of $\cE$ (well-defined only up to homotopy) is the homotopy class $[\chi]$ or, more precisely, the image of that class in the colimit
    \begin{equation*}
      \varinjlim_{\varepsilon\searrow 0}\left[\cC_{\le \varepsilon,\times}X,\mathrm{Emb}\left(r_1\to r_0\right)\right].
    \end{equation*}

  \item\label{item:def:char.cl:[[]]} The \emph{characteristic orbit} $[[\chi]]$ of $\cE$ is the orbit of the characteristic class $[\chi]$ under the composition action
    \begin{equation*}
      \left[\cC_{\le \varepsilon,\times}X,\mathrm{Emb}(r_0\to r_0)=GL\left(r_0,\Bbbk\right)\right]
      \circlearrowright
      \left[\cC_{\le \varepsilon,\times}X,\mathrm{Emb}(r_1\to r_0)\right].
    \end{equation*}
    (interpreted as before, with $\varepsilon>0$ inessential). 
  \end{enumerate}
  Note that conversely, given \Cref{eq:chi.on.cex}, we can construct a Banach bundle $\cE_{\chi}$ as in \Cref{ex:not.emb.twisted.emb}.

  All of the above extends in the obvious fashion to the equivariant setting, given a compact Lie group $\bU$ operating on all objects/morphisms involved. 
\end{definition}

\begin{remarks}\label{res:drv.spc.pcpct}
  \begin{enumerate}[(1),wide]
  \item It is implicit in \Cref{def:char.cl}, and it will be taken for granted henceforth, that all spaces involved are paracompact if $X$ is: products such as $X\times I$ (paracompact and compact: \cite[Theorem 20.12(c)]{wil_top}); more generally, products of two paracompact spaces are paracompact if one of the factors is also locally compact \cite[Theorem 2]{MR216463}; and finally, the cone $\cC X$, as is easily checked given the paracompactness of the truncations $\cC_{\ge \varepsilon}X\cong X\times I$.

  \item The action in \Cref{def:char.cl}\Cref{item:def:char.cl:[[]]} will naturally not be one by \emph{group} automorphisms (generally; it \emph{can} be, as in the proof of \Cref{cor:non.ext.spec.chi}): when $r_0=r_1$ it is simply the left-translation action. By the same token, note that in that case $[[\chi]]$ is automatically trivial (i.e. contains the constant-map class $[\chi]$). 
  \end{enumerate}
\end{remarks}

\begin{theorem}\label{th:cn.ext}
  For a paracompact $X$ $\bU$-space with compact Lie $\bU$ and a $\bU$-map \Cref{eq:chi.on.cex} the corresponding Banach bundle$_{\bU}$ $\cE_{\chi}\xrightarrowdbl{\hspace{0pt}}\cC X$ extends across the tip $0_X$ to a rank-$r_0$ (locally) trivial bundle$_{\bU}$ if and only if the characteristic orbit $[[\chi]]$ is equivariantly trivial.
\end{theorem}
\begin{proof}
  Equivariantization is simply superimposed over the trivial-action counterpart, so we cast the proof non-equivariantly.
  \begin{enumerate}[label={},wide]
  \item\textbf{($\Leftarrow$):} Trivialize $\cE|_{\cC_{\times}X}$ so as to recover $\cE=\cE_{\chi}$ for constant (i.e. trivial) $\chi$, allowing the extension of that trivialization across $\cC X$.

  \item\textbf{($\Rightarrow$):} Assume otherwise and fix a rank-$r_0$ extension $\cE'\xrightarrowdbl{\hspace{0pt}}\cC X$ of $\cE$, so that the embedding
    \begin{equation*}
      \cE^1|_{\cC_{\le \varepsilon}X}
      \lhook\joinrel\xrightarrow{\quad}
      \cE'|_{\cC_{\le \varepsilon}X}
      ,\quad
      \text{small }\varepsilon>0
    \end{equation*}
    provides a nullhomotopic continuous map $\cC_{\le \varepsilon}X\to \mathrm{Iso}_{r_1\to r_0}$ (for the domain is contractible).     That map's restriction to the punctured $\cC_{\le \varepsilon,\times}X$ need not quite have the same homotopy class as $\chi$, but the two will coincide up to changing the trivialization of the ambient $\cE^0$ over $\cC_{\le \varepsilon,\times}X$. Consequently, the left action
    \begin{equation}\label{eq:iso.act}
      \left[\cC_{\le \varepsilon,\times}X,\mathrm{Iso}_{r_0\to r_0}\right]
      \circlearrowright
      \left[\cC_{\le \varepsilon,\times}X,\mathrm{Iso}_{r_1\to r_0}\right]
    \end{equation}
    resulting from the composition
    \begin{equation*}
      \mathrm{Iso}_{r_0\to r_0}
      \times
      \mathrm{Iso}_{r_1\to r_0}
      \xrightarrow{\quad\circ\quad}
      \mathrm{Iso}_{r_1\to r_0}
    \end{equation*}
    has $[\chi]$ and the trivial class in the same orbit.
  \end{enumerate}
\end{proof}

An immediate consequence of the forward implication ($\Rightarrow$) in \Cref{th:cn.ext}:

\begin{corollary}\label{cor:non.ext.spec.chi}
  In \Cref{ex:not.emb.twisted.emb} above, if
  \begin{equation*}
    r_1>0
    ,\quad
    d-1=r_0-r_1>0
    ,\quad
    \pi_{d-1}\left(SO(r_0)\right)=\{1\}
  \end{equation*}
  and the homotopy class $[\chi]$ is non-trivial then $\cH$ does not extend across the singularity at $0\in \bD^d$ to a locally trivial rank-$r_0$ bundle $\cH'\xrightarrowdbl{\hspace{0pt}}\bD^d$.
\end{corollary}
\begin{proof}
  Given the assumption that the homotopy group $\pi_{d-1}\left(SO(r_0)\right)$ of the relevant special orthogonal group $SO(r_0)$(= the identity connected component of $\mathrm{Iso}(r_0\to r_0)$) is trivial, the left-hand acting group in \Cref{eq:iso.act} is $\bZ/2$ and that equation thus displays a $\bZ/2$-action on either $\bZ$ or $\bZ/2$ by \Cref{eq:pi.iso.stief}. Were \Cref{eq:iso.act} simply
  \begin{equation*}
    \left(\sigma:=\text{generator of }\bZ/2\right)
    \xmapsto{\quad}
    \left(
      \pi_{r_0-r_1}(V_{r_0,r_1})
      \xrightarrow{\quad-\quad}
      \pi_{r_0-r_1}(V_{r_0,r_1})
    \right),
  \end{equation*}
  the proof would be complete by \Cref{th:cn.ext}: the trivial class would then have a singleton orbit. It thus remains to settle
  
  \textbf{Claim: the \Cref{eq:iso.act} action $\bZ/2\circlearrowright \left(\bZ\text{ or }\bZ/2\right)$ is the sign change.} Recall in first instance from \cite[Theorem 25.6]{steen_fib} how the identification $\pi_{d-1=r_0-r_1}\left(V_{r_0,r_1}\right)$ functions:
  \begin{itemize}[wide]
  \item fix an $(r_1-1)$-sized orthonormal basis $(v_j)_{j=1}^{r_1-1}\subset \bR^{r_0}$ for some $(r_1-1)$-dimensional $W\le \bR^{r_0}$;

  \item whereupon a generator $s$ for $\pi_{r_0-r_1}(V_{r_0,r_1})$ is represented by
    \begin{equation*}
      \left(\bS^{r_0-r_1}:=\text{unit sphere in $W^{\perp}$}\right) \ni
      v_0
      \xmapsto{\quad}
      \left(v_j\right)_{j=0}^{r_1-1}
      \in V_{r_0,r_1}.
    \end{equation*}
  \end{itemize}
  The generator $\sigma\in \pi_{r_0-r_1}O(r_0)\cong \bZ/2$ thus acts on $ns$ by
  \begin{equation*}
    \sigma\triangleright ns
    =
    \left(
      \bS^{r_0-r_1} \ni
      v_0
      \xmapsto{\quad}
      \left(\sigma f_n(v_0),v_1\cdots v_{r_1-1}\right)
      \in V_{r_0,r_1}
    \right)
    \in
    \pi_{r_0-r_1}(V_{r_0,r_1})
  \end{equation*}
  where $\bS^{r_0-r_1}\xrightarrow{f_n}\bS^{r_0-r_1}$ represents $n=n\id \in \pi_{r_0-r_1}(\bS^{r_0-r_1})\cong \bZ$ and $\sigma$ is identified with any single determinant-$(-1)$ element of $O(r_0-r_1)$ (i.e. any element in the non-identity component), operating on $\bS^{r_0-r_1}$. This description confirms the sign-action claim, finishing the proof.
\end{proof}

While (as just discussed) the Hilbert bundle $\cH\xrightarrowdbl{\hspace{0pt}}\bD^d$ of \Cref{ex:not.emb.twisted.emb} will not extend to a locally trivial rank-$r_0$ bundle across all of $\bD^d$, it \emph{does} embed into a higher-rank locally trivial (hence trivial) bundle over $\bD^d$ by \Cref{cor:pnct.ball.emb} below. In the context of discussing various separation properties such as \emph{normality} (i.e. the $T_4$ separation property \cite[Definition 15.1]{wil_top}) or \emph{complete regularity} (\emph{$T_{3\frac 12}$} \cite[Definition 14.8]{wil_top}) we assume at least the $T_1$ axiom (points are closed) unless specified otherwise. 

\begin{proposition}\label{pr:2strata.triv.embed}
  Let $\bU$ be a compact Lie group. A subhomogeneous (F) Banach bundle$_{\bU}$ $\cE\xrightarrowdbl{\pi}X$ vanishing along a closed $\bU$-invariant subset $F\subseteq X$ and finite-type homogeneous  over the open subset $U:=X\setminus F$ assumed normal embeds into a locally trivial vector bundle$_{\bU}$. 
\end{proposition}
\begin{proof}
  Coupled with the finite-type condition, normality ensures numerability:
  \begin{itemize}[wide]
  \item A finite open cover $\cU=\left(U_i\right)_{0}^n$ of a $T_4$ space (chosen here so as to trivialize $\cE$) is well known to admit subordinate \emph{partitions of unity} in the sense of \cite[\S 13.1]{td_alg-top} in the more common non-equivariant setting (in fact, even \emph{local} finiteness suffices: \cite[Theorem 13.1.1]{td_alg-top}, \cite[\S IX.4.3, Th\'eor\`eme 3]{bourb_top_fr_5-10}, etc.).

  \item The $\bU$-equivariant extension of the preceding remark is no more difficult: the strategy adopted in \cite[Proposition I.3.12]{td_transf-gp} for arbitrary open covers of \emph{paracompact} spaces \cite[Definition 20.6]{wil_top} will do, for instance.
  \end{itemize}
  Given numerability, \cite[Proposition I.9.7]{td_transf-gp} provides a vector bundle $\cF\xrightarrowdbl{\hspace{0pt}}U$ with $\cF\oplus \cE|_U$ trivial; the conclusion follows by simply extending that trivial bundle across all of $X$.
\end{proof}

\begin{remark}\label{re:homog.obstruct}
  Note the distinction between obstructions to embedding into \emph{homogeneous} Banach bundles (typically requiring finite-type constraints \cite[Proposition I.9.7]{td_transf-gp}) and embeddability into locally trivial Banach bundles in general, with possibly infinite-dimensional fibers: the former is much more easily obstructed. Indeed, \emph{all} separable Banach bundles over compact Hausdorff spaces are embeddable in the latter, looser sense (what \cite[Corollary 2.8]{MR2120237}, to this effect, calls \emph{subtrivial}).

  The parallel dichotomy between embeddability into large locally trivial $C^*$ bundles coupled with obstructed embeddability into homogeneous ones is in evidence in \cite[Example 3.6]{bg_cx-exp}: as a subhomogeneous continuous $C^*$ field over a compact metrizable space, said example embeds into the trivial bundle with the \emph{Cuntz algebra} $\cO_2$ \cite[Example 4.6.10]{bo_cast_2008} as its fiber by \cite[Theorem A.1]{MR1461207}. Indeed, subhomogeneity ensures the requisite \emph{exactness} (in the sense of \cite[Definition 2,3.2]{bo_cast_2008}) by \cite[Proposition 2.7.7]{bo_cast_2008}.
\end{remark}

\begin{corollary}\label{cor:pnct.ball.emb}
  Any subhomogeneous (F) Banach bundle $\cE\xrightarrowdbl{\hspace{0pt}}\bD$ over a ball with a single dimension jump at $0\in \bD$ and of finite type over $\bD^{\times}:=\bD\setminus\{0\}$ embeds into a trivial homogeneous one. 
\end{corollary}
\begin{proof}
  Assuming the fiber $\cE_0$ $d$-dimensional, $d$ linearly-independent sections extend locally around $0$ \cite[p.15]{dg_ban-bdl} and in fact globally: simply observe that any scaling
  \begin{equation*}
    \bD^d_{\le \varepsilon}
    \xrightarrow[\quad\cong\quad]{\quad\varphi:=\frac{\bullet}{\varepsilon}\quad}
    \bD^d
  \end{equation*}
  implements a Banach-bundle isomorphism between $\cE_{\bD^d_{\le \varepsilon}}$ and the pullback (what \cite[\S II.13.3]{fd_bdl-1} calls the \emph{retraction}) $\varphi^* \cE$. This provides an embedding $\bD\times \Bbbk^d \cong \cF \le \cE$ of Banach bundles, $\Bbbk\in \left\{\bR,\bC\right\}$ denoting the ground field.
  
  The paracompactness of $\bD^{\times}$ ensures \cite[Lemma 2]{zbMATH03179258} the existence of an inner product on $\cE|_{\bD^{\times}}$, whence a decomposition
  \begin{equation*}
    \cE
    \cong
    \cF\oplus \cF^{\perp}
    ,\quad
    \cF^{\perp}_0=\{0\}. 
  \end{equation*}
  The conclusion follows from an application of \Cref{pr:2strata.triv.embed} to $\cF^{\perp}$ with $X:=\bD$ and $F:=\left\{0\right\}$.
\end{proof}

\begin{remarks}\label{res:phl.t312.t4}
  \begin{enumerate}[(1),wide]
  \item\label{item:res:phl.t312.t4:gap} In reference to the possibility of complementing a vector bundle $\cE$ with an $\cF$ so that $\cE\oplus \cF$ is trivial (as was needed in the above proof of \Cref{pr:2strata.triv.embed}), \cite[Lemma 2.6]{zbMATH05172034} asserts among other things that (non-equivariantly) over a $T_{3\frac 12}$ base space
    \begin{enumerate}[(a),wide]
    \item\label{item:res:phl.t312.t4:ft} $\cE$ is of finite type;
    \item\label{item:res:phl.t312.t4:cmplmt} if and only if there is a vector bundle $\cF$ over $X$ with $\cE\oplus \cF$ trivial. 
    \end{enumerate}
    The proof appears to me to contain a small gap. The issue is with that discussion's reference to \cite[Proposition 3.5.8]{hus_fib}, where the same equivalence appears as (1) $\Leftrightarrow$ (3): while the leftward implication (1) $\Leftarrow$ (3) (i.e. \Cref{item:res:phl.t312.t4:ft} $\Leftarrow$ \Cref{item:res:phl.t312.t4:cmplmt}) does go through unconditionally, the converse in turn appeals to \cite[Theorem 3.5.5]{hus_fib}, where paracompactness is assumed (though \cite[Proposition 3.5.8]{hus_fib} simply refers to ``spaces'').
    
    An examination of the proof of \cite[Theorem 3.5.5]{hus_fib} indicates that the crucial ingredient in confirming \Cref{item:res:phl.t312.t4:ft} $\Rightarrow$ \Cref{item:res:phl.t312.t4:cmplmt} is the existence of a partition of unity subordinate to a cover $\cU=\left(U_i\right)_{0}^n$ trivializing $\cE$. That cover being assumed finite, normality will suffice (as noted in the proof of \Cref{pr:2strata.triv.embed}). Given the strict strength ordering
    \begin{equation*}
      T_{3\frac 12}
      \quad
      \overset{\text{\cite[Corollary 15.7]{wil_top}}}{\prec}
      \quad
      T_4
      \quad
      \overset{\text{\cite[Theorem 20.10]{wil_top}}}{\prec}
      \quad
      \text{paracompact}
    \end{equation*}
    (\cite[Examples 82 and 143]{ss_countertop} for strictness), it is unclear whether the $T_{3\frac 12}$ condition assumed in \cite[Lemma 2.6]{zbMATH05172034} suffices.

  \item\label{item:res:phl.t312.t4:plus.numrbl} What is in any case true is that \Cref{item:res:phl.t312.t4:cmplmt} holds if and only if \Cref{item:res:phl.t312.t4:ft} does in conjunction with numerability: trivializability over a finite open cover admitting a subordinate partition of unity. This is precisely what \cite[Proposition I.9.7]{td_transf-gp} proves (equivariantly). 
    
  \item\label{item:res:phl.t312.t4:nrm.equiv} Bolstering the above-noted apparent necessity of normality, note that the existence of partitions of unity subordinate to arbitrary
    \begin{itemize}[wide]
    \item locally-finite;
    \item equivalently, finite;
    \item equivalently, binary
    \end{itemize}
    open covers in fact \emph{characterizes} normal spaces (among those assumed $T_1$, as we always do here barring disavowals). One implication is the already-cited \cite[Theorem 13.1.1]{td_alg-top}; for the converse, simply observe that
    \begin{itemize}[wide]
    \item a partition of unity $\left(\varphi_{0,1}\right)$ subordinate to an open cover $X=U_0\cup U_1$ provides an open cover
      \begin{equation*}
        X
        =
        V_0\cup V_1
        ,\quad
        V_i:=\varphi_i^{-1}\left(\bR_{>0}\right)
      \end{equation*}
      with $\ol{V_i}\subseteq U_i$, $i=0,1$;

    \item which precisely translates to the formulation \cite[Lemma 31.1(b)]{mnk} of normality requiring that for arbitrary
      \begin{equation*}
        \left(\text{closed }F_0:=X\setminus U_0\right)
        \subseteq
        \left(\text{open }U_1\right)
      \end{equation*}
      there be a neighborhood $V_1\supseteq F_0$ with $\ol{V_1}\subseteq U_1$.
    \end{itemize}
  \end{enumerate}
\end{remarks}

The device employed in the proof of \Cref{cor:pnct.ball.emb}, extending the origin fiber to a locally trivial subbundle, is one instance of a broader principle we make a detour to isolate as \Cref{th:ext.loc.triv.hmtp.eq}.

\pf{th:ext.loc.triv.hmtp.eq}
\begin{th:ext.loc.triv.hmtp.eq}
  In the specific case of \Cref{item:th:ext.loc.triv.hmtp.eq:ban}:
  \begin{itemize}[wide]
  \item Effect the desired extension \emph{locally} to $\ol{\cF}|_{U}$ around $F$ by \cite[Theorem 0.2(1)]{MR5060733}, on a paracompact neighborhood $U\supseteq F$.

  \item $U$ and $X$ may be assumed equivariantly homotopic through maps identical on $F$ by hypothesis, so that $X\setminus F$ and $U\setminus F$ are also equivariantly homotopic.

  \item Focusing on loci where the bundles are homogeneous, note that the embedding $\ol{\cF}\le \cE$ over $U\setminus F$ amounts to \emph{reducing the unitary structure group} \cite[Definition 12.1]{MR3331607} $\bU(\dim \cE_x)$ of the principal bundle attached to $\cE$ to a subgroup $\bU(d')\times \bU(d'')\le \bU(d:=\dim \cE_x)$ for $d'+d''=d$.

  \item Such a reduction is a homotopy invariant, by the classification \cite[Theorems I.8.12 and I.8.15]{td_transf-gp} of equivariant bundles by means of maps into equivariant classifying spaces.

  \item If possible for $U\setminus F$, it must thus be possible for the equivariantly homotopic $X\setminus F$. 
  \end{itemize}
  The other items run in parallel, appealing to \cite[Theorem 0.2(2) and Theorem 0.3]{MR5060733} (the former for \Cref{item:th:ext.loc.triv.hmtp.eq:hilb} and the latter's two items for \Cref{item:th:ext.loc.triv.hmtp.eq:balg} and \Cref{item:th:ext.loc.triv.hmtp.eq:cast} respectively).
\end{th:ext.loc.triv.hmtp.eq}

\begin{remarks}\label{res:post.ext.emb}
  \begin{enumerate}[(1),wide]
  \item Note the necessity for \emph{some} constraints in the spirit of the assumed homotopy equivalence in \Cref{th:ext.loc.triv.hmtp.eq}: as observed in \cite[Remark 1.2(2)]{MR5060733}, $\cF\xrightarrowdbl{}F$ need not extend globally at all (e.g. if non-trivial with $X$ supporting only trivial bundles).

  \item\label{item:res:post.ext.emb:inv.max.cpct} There are some subtleties to contend with in the proof of \Cref{th:ext.loc.triv.hmtp.eq} in ensuring equivariant soundness, as the present observation elaborates. 

    What justifies working with \emph{compact} structure Lie groups in treating vector or matrix-algebra bundles (e.g. \cite[Assertion 18.3.4]{hjjm_bdle}, identifying principal $GL(n)$-bundles with $U(n)$-bundles) is the fact that embeddings of maximal compact subgroups of Lie groups with finitely many components are homotopy equivalences \cite[Theorem C]{zbMATH00540631}. In the context of the proof above, a vector $d'$-subbundle of a $d$-bundle means in first instance restricting the structure group to the isotropy group of a $d'$-subspace of $\bC^d$ in $\bU(d)$, which is not itself compact; therefrom, one restricts to a maximal compact subgroup invoking the same homotopy-equivalence principle.

    That this all goes through as just described $\bU$-equivariantly follows from the fact that a compact group acting on an \emph{almost-connected} (compact quotient space of components) locally compact group leaves some maximal compact subgroup invariant. This is recorded in \cite[Proposition 1.18]{2506.09642v1} with \cite[Remark 1.19]{2506.09642v1} pointing to prior sources; while presumably well-known, the statement seems to be otherwise somewhat difficult to locate in the literature in this generality.

  \item Something in the nature of the preceding remark seems to be taken for granted to a certain extent when handling maximal-compact reduction in \cite[\S 12]{MR3331607}: that paper assumes \emph{trivial} $\bU$-actions on structure groups, but \cite[Introduction]{MR3331607} asserts the material extensible to the more general setup of \cite[\S I.8]{td_alg-top}. 
  \end{enumerate}
\end{remarks}

\Cref{th:can.emb.ban.hilb} drops the vanishing constraint in \Cref{pr:2strata.triv.embed} at the cost of various topological adjustments (the homotopy equivalence avoiding the likes of \Cref{ex:pnct.sph}). 

\pf{th:can.emb.ban.hilb}
\begin{th:can.emb.ban.hilb}
  In both cases extend $\cF:=\cE|_{F}$ globally to a locally trivial $\ol{\cF}\le \cE$ by \Cref{th:ext.loc.triv.hmtp.eq}\Cref{item:th:ext.loc.triv.hmtp.eq:ban} and \Cref{item:th:ext.loc.triv.hmtp.eq:hilb} respectively, consider the orthogonal decomposition
  \begin{equation*}
    \cE=\ol{\cF}\oplus \cG
    ,\quad
    \cG|_{F}=0
  \end{equation*}
  with respect to a Hermitian structure on $\cE$ over $X\setminus F$, and embed both summands into trivial bundles: $\ol{\cF}$ by local triviality and $\cG$ by \Cref{pr:2strata.triv.embed} (given its vanishing over $F$).
\end{th:can.emb.ban.hilb}

\Cref{cor:pnct.ball.emb} does require some of the specifics of the situation it addresses: it is not difficult to alter \Cref{ex:not.emb.twisted.emb} ever so slightly so as to ensure non-embeddability into locally trivial vector bundles regardless of rank.

\begin{example}\label{ex:pnct.sph}
  The only alteration to \Cref{ex:not.emb.twisted.emb} will be to substitute for \Cref{eq:h0dd} the (still trivial) rank-$r_0$ $\cH^0$ over a punctured \emph{sphere}
  \begin{equation*}
    \bS^d_{\times}
    :=
    \bD^d_{\times}/\partial \bD^d
  \end{equation*}
  obtained by collapsing the boundary of $\bD^d$ (a $(d-1)$-sphere) to a point. The result of gluing the trivial
  \begin{equation*}
    \cH^0
    \xrightarrowdbl{\quad}
    \bS_{\times}^d
    =
    \bS^d\setminus \{p\}
    \quad\text{and}\quad
    \cH^1
    \xrightarrowdbl{\quad}
    \bD_{\le R}(p)
    :=
    \text{radius-$R$ spherical ball around $p$}
  \end{equation*}
  of respective ranks $r_0$ and $r_1$ along a continuous map
  \begin{equation*}
    \bS_{\times}^d
    \cap
    \bD_{\le R}(p)
    =
    \bD_{\le R}(p)\setminus\{p\}
    \xrightarrow{\quad\chi\quad}
    \mathrm{Iso}_{r_1\to r_0}
  \end{equation*}
  will now be a Hilbert bundle $\cH\xrightarrowdbl{\hspace{0pt}}\bS^d$ over the full sphere, with $r_0$-dimensional generic fiber and $r_1$-dimensional fiber at an exceptional point $p\in \bS^d$, image of the usual origin $0\in \bD^d$ through the boundary collapse $\bD^d\xrightarrowdbl{\hspace{0pt}}\bS^d$. It follows from \Cref{cor:susp.prcmpct} with $X:=\bS^{d-1}$ that $\cH$ cannot embed into a locally trivial vector bundle over $\bS^d$ unless $\chi$ is nullhomotopic.
\end{example}

For a topological space $X$ write
\begin{equation*}
  \Sigma X := \cC_+X\sqcup_{X\times \{0\}}\cC_-X
  ,\quad
  \cC_{\bullet}X
  :=
  X\times I_{\bullet}/X\times\left\{\bullet\right\}
  ,\quad
  \bullet\in \left\{\pm 1\right\}
\end{equation*}
for the \emph{suspension} \cite[Example 9.12(f)]{wil_top} of $X$, where $I_{\bullet}$ is the interval with endpoints $\left\{\bullet,0\right\}$ respectively (so that $\cC_{\bullet}X$ are copies of the cone on $X$). The notation $\cE_{\chi}$ of \Cref{def:char.cl} extends in obvious fashion to bundles over $\Sigma X$, with the exceptional fiber being by convention that at the tip $0_-\in \Sigma X$, image of $X\times\{-1\}$. 

\begin{corollary}\label{cor:susp.prcmpct}
  Let $X$ be a paracompact space and $\cE=\cE_{\chi}\xrightarrowdbl{\hspace{0pt}}\Sigma X$ the continuous Banach bundle over $\Bbbk\in \left\{\bR,\bC\right\}$ attached to some $\chi$ as in \Cref{eq:chi.on.cex} over the negative punctured cone $\cC_{-,\le \varepsilon,\times} X$.  

  If $\cE$ embeds into a locally trivial homogeneous bundle over $\Sigma X$ then $[[\chi]]$ is trivial. 
\end{corollary}
\begin{proof}
  Consider such an embedding $\cE\le \cE'$, with the ambient $\cE'$ of rank $r$ and trivial (which we may as well assume). The embedding $\cE^0\le \cE'|_{\Sigma_{\times} X}$ gives a continuous map
  \begin{equation*}
    \Sigma X\setminus \left\{0_-\right\}
    =:
    \Sigma_{\times} X
    \xrightarrow{\quad}
    \bG\left(r_0,r\right)
  \end{equation*}
  into the Grassmannian of $r_0$-planes in $\bR^r$; the domain being contractible, that map can be homotoped into a constant one without affecting the analogous 
  \begin{equation*}
    \cC_{-,\le \varepsilon}X
    \xrightarrow{\quad}
    \bG\left(r_1,r\right)
  \end{equation*}
  (which we can assume constant, having fixed a trivialization of $\cE^1$). This affords the extension of $\cE^0$ across all of $\Sigma X$ to a (trivial) rank-$r_0$ subbundle of the ambient $\cE'$, so we are back within the scope of \Cref{ex:not.emb.twisted.emb}.
\end{proof}

%%%%%%%%%%%%%%%%%%%%%%%%%%%%%%%%
%%%%%%%%%%%%%%%%%%%%%%%%%%%%%%%%
\section{On and around $C^*$-bundle homogeneity}\label{se:cast.emb}

The $C^*$-bundle version of \Cref{th:can.emb.ban.hilb} will not quite go through. In reference to \Cref{pr:cast.need.trace} below attesting this, recall \cite[Definition II.6.10.1 and Theorem II.6.10.2]{blk_oa} that an \emph{expectation} $A\xrightarrowdbl{E}B$ of a $C^*$-algebra onto a $C^*$-subalgebra is a norm-1 idempotent (or rather the corestriction of one); it will then automatically be completely positive and a $B$-bimodule morphism. 

Much is made below of (equivariant) \emph{numerability} for locally trivial bundles$_{\bU}$ $\cE\xrightarrowdbl{}X$ in the sense of \cite[(I.6.4), (I.6.5), post Lemma I.8.9 and Definition I.9.1]{td_alg-top}:
\begin{itemize}[wide]
\item equivariantly \emph{trivializable} over $\cU=\left(U_j\right)_j$ of an open cover $X=\bigcup_j U_j$, i.e. with the restrictions $\cE|_{U_j}$ admitting bundle$_{\bU}$ maps into \emph{local objects} (bundles of the form $\bU\times_{\bK}V$ for $\bK\le \bU$ and a $\bK$-representation $V$ in the case of vector bundles, analogues for matrix bundles, etc.);
\item with $U_j$ admitting a $\bU$-invariant subordinate partition of unity. 
\end{itemize}
\emph{Finite} numerability simply means that the cover $\cU$ can be chosen finite.

\begin{proposition}\label{pr:cast.need.trace}
  \begin{enumerate}[(1),wide]
  \item\label{item:pr:cast.need.trace:prcpct} Let $\bU$ be a compact Lie group acting on the paracompact space $X$. The bounded-section $C^*$-algebra $A:=\Gamma_b(\cA)$ of a subhomogeneous continuous $C^*$ bundle$_{\bU}$ $\cA\xrightarrowdbl{}X$ equivariantly embeddable into a locally trivial one admits a $\bU$-equivariant tracial faithful expectation
    \begin{equation*}
      A
      \xrightarrowdbl{\quad E\quad}
      C_b(X)
      :=
      \text{algebra of complex bounded functions on $X$}.
    \end{equation*}

  \item\label{item:pr:cast.need.trace:nrml} If $\cA$ is of finite type it suffices to assume $X$ normal. 
  \end{enumerate}
\end{proposition}
\begin{proof}

  \begin{enumerate}[(1),wide]
  \item Consider an embedding
    \begin{equation*}
      \left(\cA\xrightarrowdbl{\quad} X\right)
      \lhook\joinrel\xrightarrow{\quad}
      \left(\cB\xrightarrowdbl{\quad} X\right)
    \end{equation*}
    into a locally trivial subhomogeneous $C^*$ bundle. The latter's $C^*$-algebra of bounded sections carries the canonical (tracial, \emph{finite-index} \cite[Definition 2]{fk_fin-ind}) expectation $\Gamma_b(\cB)=:B\xrightarrowdbl{T}C_b(X)$ constructed in \cite[Proposition 3.4]{bg_cx-exp} (over compact Hausdorff $X$, but the construction goes through in the present generality). The canonical nature of that construction also ensures $\bU$-equivariance, and this suffices for \cite[Corollary 3.5]{bg_cx-exp} (also stated there for compact Hausdorff base spaces) to go through: restrict that expectation back to $A$ along $A\lhook\joinrel\to B$.

  \item All of the above goes through so long as we have trivialized the bundle over a numerable cover; as recalled in \Cref{res:phl.t312.t4}\Cref{item:res:phl.t312.t4:nrm.equiv}, a space is normal precisely when its finite (or equivalently, binary) open covers are numerable. 
  \end{enumerate}
\end{proof}

The reason why, as anticipated, \Cref{pr:cast.need.trace} rules out transporting \Cref{th:can.emb.ban.hilb} straightforwardly to the $C^*$ setting is that there are simple examples of continuous $C^*$ bundles meeting all of the topological requirements but failing to carry the necessary tracial expectation: \cite[Example 2.2]{CHIRVASITU2026130746} gives one such bundle over $[-1,1]$ for instance, with generic fiber $M_3$ away from 0 and $\bC^2$ at 0.

\begin{remark}\label{re:emb.in.triv.mn}

  In reference to \Cref{pr:cast.need.trace}\Cref{item:pr:cast.need.trace:nrml}, recall that locally trivial matrix-algebra bundles over normal spaces always embed into \emph{trivial} such: this is the multiplicative analogue of \cite[Proposition I.9.7]{td_transf-gp}, and is proven as part of \cite[Proposition 2.9]{zbMATH05172034} subject to the caveat raised in \Cref{res:phl.t312.t4}\Cref{item:res:phl.t312.t4:gap}: the cited result assumes only the $T_{3\frac 12}$ separation axiom, whereas normality seems to be required (or at least not obviously removable).
\end{remark}

It will be of some use to record the following consequence of \Cref{pr:cast.need.trace}'s proof (though not quite of its statement).

\begin{proposition}\label{pr:emb.mtrx}
  Let $\bU$ be a compact Lie group.
  \begin{enumerate}[(1),wide]
  \item\label{item:pr:emb.mtrx:emb.mat} A numerable locally trivial subhomogeneous $C^*$ bundle$_{\bU}$ over a $T_{3\frac 12}$ space $X$ embeds into a locally trivial matrix bundle$_{\bU}$.

  \item\label{item:pr:emb.mtrx:fin} If the original bundle$_{\bU}$ is finitely numerable the matrix bundle$_{\bU}$ can be chosen so as well. 
  \end{enumerate}  
\end{proposition}
\begin{proof}
  \begin{enumerate}[label={},wide]
  \item\textbf{\Cref{item:pr:emb.mtrx:emb.mat}} The canonical expectation $\Gamma_b(\cA)\xrightarrowdbl{E}C_b(X)$ of \Cref{pr:cast.need.trace} affords faithful GNS representations $\cA_x\lhook\joinrel\to \cL(\cA_x)$ at each $x\in X$: $\cA_x$ is here regarded as a Hilbert space equipped with its structure induced by the tracial state $\cA_x\xrightarrow{E_x}\bC$. Said GNS representations will glue to the desired embedding, equivariant because everything else is ($E$ inclusive, given its canonical nature).

  \item\textbf{\Cref{item:pr:emb.mtrx:fin}} The preceding argument does not disturb finite numerability. 
  \end{enumerate}
\end{proof}

\cite[Proposition 2.9, (1) $\Leftrightarrow$ (2)]{zbMATH05172034} reduces the finite-type condition for matrix bundles to the same for the underlying vector bundle; the canonical-expectation machinery developed in \cite{bg_cx-exp} and employed above disposes of that reduction (equivariantly) fairly quickly, per \Cref{th:ft.emb.ft}. 

\begin{theorem}\label{th:ft.emb.ft}
  The following conditions on a locally trivial subhomogeneous $C^*$ bundle$_{\bU}$ $\cA\xrightarrowdbl{}X$ over a normal $\bU$-space are equivalent. 
  \begin{enumerate}[(a),wide]
  \item\label{item:th:ft.emb.ft:ft} $\cA$ is finite-type.

  \item\label{item:th:ft.emb.ft:emb.ft.mtrx} $\cA$ embeds into a finite-type locally trivial subhomogeneous matrix bundle$_{\bU}$.
    
  \item\label{item:th:ft.emb.ft:emb.ft} $\cA$ embeds into a finite-type locally trivial subhomogeneous $C^*$ bundle$_{\bU}$.

  \item\label{item:th:ft.emb.ft:ft.vb} $\cA$ is finite-type as a vector bundle$_{\bU}$.
  \item\label{item:th:ft.emb.ft:emb.ft.vb} $\cA$ embeds into a finite-type vector bundle$_{\bU}$. 
  \end{enumerate}
\end{theorem}
\begin{proof}
  Condition \Cref{item:th:ft.emb.ft:ft} is formally stronger than all others except for \Cref{item:th:ft.emb.ft:emb.ft.mtrx}, which \Cref{pr:emb.mtrx}\Cref{item:pr:emb.mtrx:fin} handles instead (along with \Cref{item:th:ft.emb.ft:emb.ft.mtrx} $\Leftrightarrow$ \Cref{item:th:ft.emb.ft:emb.ft}); thus: \Cref{item:th:ft.emb.ft:ft} implies everything else.

  \begin{enumerate}[label={},wide]
  \item \textbf{\Cref{item:th:ft.emb.ft:emb.ft.mtrx} $\Rightarrow$ \Cref{item:th:ft.emb.ft:ft}:} Consider an embedding $\cA\lhook\joinrel\to \cB$ as stated, with $\cB$ trivialized by a finite open $\bU$-cover $\cU=(U_j)_j$. Shrinking $U_j$ to an open $\bU$-cover $\cV=\left(V_j\right)_{j}$, $\overline{V_j}\subseteq U_j$ (as always possible on a normal space \cite[proof of Theorem 36.1]{mnk}), we can restrict attention to a single (again normal) $\overline{V_j}$ and assume embeddability into a \emph{trivial} (rather than just finite-type) bundle$_{\bU}$.
    
    The conclusion now follows as in the argument supporting \cite[Proposition I.9.7]{td_transf-gp}: an embedding $\cA\lhook\joinrel\to X\times M_n$ realizes $\cA$ as a pullback of the tautological bundle over the $\bU$-manifold of $C^*$-subalgebras of $M_n$, and said bundle is of finite type (the base space being a compact $\bU$-manifold).
    
  \item \textbf{\Cref{item:th:ft.emb.ft:emb.ft.vb} $\Rightarrow$ \Cref{item:th:ft.emb.ft:ft.vb}:} entirely analogous to the preceding argument, so omitted. 

  \item \textbf{\Cref{item:th:ft.emb.ft:ft.vb} $\Rightarrow$ \Cref{item:th:ft.emb.ft:ft}:} Consider the expectation $\Gamma_b(\cA)\xrightarrowdbl{E}C_b(X)$ of \Cref{pr:cast.need.trace}, with its attendant $X$-global GNS representation
    \begin{equation}\label{eq:a2enda}
      \left(\cA\xrightarrowdbl{\quad}X\right)
      \lhook\joinrel\xrightarrow[\quad]{\quad\text{$C^*$-bundle$_{\bU}$ morphism}\quad}
      \left(\cE nd~\cA\xrightarrowdbl{\quad}X\right).
    \end{equation}
    The latter bundle$_{\bU}$ has finite type by assumption (for $\cA$ does, as a vector bundle$_{\bU}$). That $\cA$ is finite-type as a $C^*$ bundle$_{\bU}$ now follows from the already-settled implication \Cref{item:th:ft.emb.ft:emb.ft.mtrx} $\Rightarrow$ \Cref{item:th:ft.emb.ft:ft}.
  \end{enumerate}
\end{proof}

Another aspect of \cite[Proposition 2.9]{zbMATH05172034} is its characterization of finite-type matrix bundles as precisely those that extend across the \emph{Stone-\v{C}ech compactification} \cite[Definition 19.4]{wil_top} $X\lhook\joinrel\to \beta X$. This too has its equivariant counterpart, valid moreover for arbitrary locally trivial subhomogeneous $C^*$ bundles$_{\bU}$. In preparation for the statement, we remind the reader that $T_{3\frac 12}$ spaces equipped with an action by a compact (indeed, locally compact, perhaps non-Lie) group $\bU$ admit universal \emph{equivariant compactifications} $X\lhook\joinrel\to \beta_{\bU}X$ (\cite[Proposition 3.1]{dvr-ex}, \cite[\S 2.4]{dvr_puc_1977}, the recent \cite{megr_max-equiv-cpct} and its numerous references, etc.): $\beta_{\bU}$ is a left adjoint to the inclusion functor
\begin{equation*}
  \left(\text{compact $T_2$ $\bU$-spaces}\right)
  \lhook\joinrel\xrightarrow{\quad}
  \left(\text{$T_{3\frac 12}$ $\bU$-spaces}\right),
\end{equation*}
with the components $X\lhook\joinrel\to \beta_{\bU}X$ of the adjunction's unit embeddings recovering $X$ with the subspace topology.

\begin{theorem}\label{th:ft.ext.cpct}
  Let $\bU$ be a compact Lie group and $X$ a normal $\bU$-space. The following conditions on a locally trivial subhomogeneous $C^*$ bundle$_{\bU}$ $\cA\xrightarrowdbl{}X$ are equivalent.
  \begin{enumerate}[(a),wide]

  \item\label{item:th:ft.ext.cpct:some.mfld} $\cA$ pulls back from an equivariant map to a compact $C^{\infty}$ $\bU$-manifold.
    
  \item\label{item:th:ft.ext.cpct:some.cpct} $\cA$ pulls back from an equivariant map to a compact Hausdorff $\bU$-space.
    
  \item\label{item:th:ft.ext.cpct:beta} $\cA$ extends across the universal $\bU$-equivariant compactification $X\lhook\joinrel\to \beta_{\bU}X$.

  \item\label{item:th:ft.ext.cpct:some.cpctf} $\cA$ extends across some $\bU$-equivariant compactification $X\lhook\joinrel\to Y$.

  \item\label{item:th:ft.ext.cpct:ft} $\cA$ is of finite type. 
  \end{enumerate}
\end{theorem}
\begin{proof}
  \Cref{item:th:ft.ext.cpct:some.mfld} $\Rightarrow$ \Cref{item:th:ft.ext.cpct:some.cpct} and \Cref{item:th:ft.ext.cpct:beta} $\Rightarrow$ \Cref{item:th:ft.ext.cpct:some.cpctf} are purely formal, for \Cref{item:th:ft.ext.cpct:some.cpct} $\Rightarrow$ \Cref{item:th:ft.ext.cpct:beta} pull back along the right-hand map in the diagram
  \begin{equation*}
    \begin{tikzpicture}[>=stealth,auto,baseline=(current  bounding  box.center)]
      \path[anchor=base] 
      (0,0) node (l) {$X$}
      +(2,.5) node (u) {$\beta_{\bU}X$}
      +(4,0) node (r) {$Y$}
      ;
      \draw[right hook->] (l) to[bend left=6] node[pos=.5,auto] {$\scriptstyle $} (u);
      \draw[->] (u) to[bend left=6] node[pos=.5,auto] {$\scriptstyle $} (r);
      \draw[->] (l) to[bend right=6] node[pos=.5,auto,swap] {$\scriptstyle $} (r);
    \end{tikzpicture}
  \end{equation*}
  and \Cref{item:th:ft.ext.cpct:some.cpctf} $\Rightarrow$ \Cref{item:th:ft.ext.cpct:ft} by restricting a finite trivializing cover from $Y$ to $X$. It is \Cref{item:th:ft.ext.cpct:ft} $\Rightarrow$ \Cref{item:th:ft.ext.cpct:some.mfld}, then, that carries the payload and will occupy the rest of the proof. 

  Note in first instance that we again have available the embedding \Cref{eq:a2enda}, realizing $\cA$ as a $C^*$ subbundle$_{\bU}$ of $\cE nd(\cA)$ for $\cA$ in the latter expression regarded as a finite-type vector bundle$_{\bU}$. Precisely as in \cite[Proposition I.9.7]{td_transf-gp}, the latter finite-type assumption provides
  \begin{equation*}
    \left(\cA\xrightarrowdbl{\quad}X\right)
    \lhook\joinrel\xrightarrow[\quad]{\quad\text{Hilbert bundle$_{\bU}$ embedding}\quad}
    X\times V
  \end{equation*}
  for a unitary $\bU$-representation $V$, whence an identification of each fiber $\cA_x$ with a subspace of the fixed $V$. This gives a map
  \begin{equation*}
    X\ni x
    \xmapsto{\quad}
    \cA_x
    \in
    \bG_{C^*}(\End V)
    :=
    \text{non-unital Grassmannian of $C^*$-subalgebras}
  \end{equation*}
  realizing $\cA$ as a pullback along it. The target being a $\bU$-manifold, this is the desired conclusion. 
\end{proof}

%%%%%%%%%%%%%%%%%%%%%%%%%%%%%%%% 
%%%%%%%%%%%%%%%%%%%%%%%%%%%%%%%%

\addcontentsline{toc}{section}{References}
%\bibliography{bib}{}

\begin{thebibliography}{10}

\bibitem{blk_oa}
B.~Blackadar.
\newblock {\em Operator algebras}, volume 122 of {\em Encyclopaedia of
  Mathematical Sciences}.
\newblock Springer-Verlag, Berlin, 2006.
\newblock Theory of $C^*$-algebras and von Neumann algebras, Operator Algebras
  and Non-commutative Geometry, III.

\bibitem{MR1461207}
Etienne Blanchard.
\newblock Subtriviality of continuous fields of nuclear {$C^*$}-algebras.
\newblock {\em J. Reine Angew. Math.}, 489:133--149, 1997.

\bibitem{bg_cx-exp}
Etienne Blanchard and Ilja Gogi\'{c}.
\newblock On unital {$C(X)$}-algebras and {$C(X)$}-valued conditional
  expectations of finite index.
\newblock {\em Linear Multilinear Algebra}, 64(12):2406--2418, 2016.

\bibitem{MR2120237}
Etienne Blanchard and Eberhard Kirchberg.
\newblock Global {G}limm halving for {$C^*$}-bundles.
\newblock {\em J. Operator Theory}, 52(2):385--420, 2004.

\bibitem{bourb_top_fr_5-10}
N.~Bourbaki.
\newblock {\em \'{E}l\'{e}ments de math\'{e}matique. {T}opologie
  g\'{e}n\'{e}rale. {C}hapitres 5 \`a 10}.
\newblock Hermann, Paris, 1974.

\bibitem{btd_lie_1995}
Theodor Br{\"o}cker and Tammo tom Dieck.
\newblock {\em Representations of compact {Lie} groups. {Corrected} reprint of
  the 1985 orig}, volume~98 of {\em Grad. Texts Math.}
\newblock New York, NY: Springer, corrected reprint of the 1985 orig. edition,
  1995.

\bibitem{bo_cast_2008}
Nathanial~P. Brown and Narutaka Ozawa.
\newblock {\em {$C^*$}-algebras and finite-dimensional approximations},
  volume~88 of {\em Graduate Studies in Mathematics}.
\newblock American Mathematical Society, Providence, RI, 2008.

\bibitem{bbi}
Dmitri Burago, Yuri Burago, and Sergei Ivanov.
\newblock {\em A course in metric geometry}, volume~33 of {\em Graduate Studies
  in Mathematics}.
\newblock American Mathematical Society, Providence, RI, 2001.

\bibitem{2506.09642v1}
Alexandru Chirvasitu.
\newblock Pervasive ellipticity in locally compact groups, 2025.
\newblock \url{http://arxiv.org/abs/2506.09642v1}.

\bibitem{MR5060733}
Alexandru Chirvasitu.
\newblock Equivariant {B}anach-bundle germs.
\newblock {\em Topology Appl.}, 385:Paper No. 109821, 2026.

\bibitem{CHIRVASITU2026130746}
Alexandru Chirvasitu.
\newblock Non-commutative branched covers and bundle unitarizability.
\newblock {\em Journal of Mathematical Analysis and Applications}, page 130746,
  2026.

\bibitem{dvr_puc_1977}
Jan de~Vries.
\newblock Equivariant embeddings of {$G$}-spaces.
\newblock In {\em General topology and its relations to modern analysis and
  algebra, {IV} ({P}roc. {F}ourth {P}rague {T}opological {S}ympos., {P}rague,
  1976), {P}art {B}}, pages 485--493, 1977.

\bibitem{dvr-ex}
Jan de~Vries.
\newblock On the existence of {$G$}-compactifications.
\newblock {\em Bull. Acad. Polon. Sci. S\'{e}r. Sci. Math. Astronom. Phys.},
  26(3):275--280, 1978.

\bibitem{dg_ban-bdl}
M.~J. Dupr{\'e} and R.~M. Gillette.
\newblock {\em Banach bundles, {Banach} modules and automorphisms of
  {{\(C^*\)}}-algebras}, volume~92 of {\em Res. Notes Math., San Franc.}
\newblock Pitman Publishing, London, 1983.

\bibitem{fell_struct}
J.~M.~G. Fell.
\newblock The structure of algebras of operator fields.
\newblock {\em Acta Math.}, 106:233--280, 1961.

\bibitem{fell_ind-ban-bdl}
J.~M.~G. Fell.
\newblock {\em Induced representations and {Banach} {{\(^*\)}}-algebraic
  bundles. {With} an appendix due to {A}. {Douady} and {L}. {Dal}
  {Soglio}-{Herault}}, volume 582 of {\em Lect. Notes Math.}
\newblock Springer, Cham, 1977.

\bibitem{fd_bdl-1}
J.~M.~G. Fell and R.~S. Doran.
\newblock {\em Representations of *-algebras, locally compact groups, and
  {Banach} *- algebraic bundles. {Vol}. 1: {Basic} representation theory of
  groups and algebras}, volume 125 of {\em Pure Appl. Math., Academic Press}.
\newblock Boston, MA etc.: Academic Press, Inc., 1988.

\bibitem{fk_fin-ind}
Michael Frank and Eberhard Kirchberg.
\newblock On conditional expectations of finite index.
\newblock {\em J. Oper. Theory}, 40(1):87--111, 1998.

\bibitem{gog_top-fg}
Ilja Gogi{\'c}.
\newblock Topologically finitely generated {Hilbert} {{\(C(X)\)}}-modules.
\newblock {\em J. Math. Anal. Appl.}, 395(2):559--568, 2012.

\bibitem{hatch_at}
Allen Hatcher.
\newblock {\em Algebraic topology}.
\newblock Cambridge: Cambridge University Press, 2002.

\bibitem{zbMATH00540631}
Karl~H. Hofmann and Christian Terp.
\newblock Compact subgroups of {Lie} groups and locally compact groups.
\newblock {\em Proc. Am. Math. Soc.}, 120(2):623--634, 1994.

\bibitem{hk_shv-bdl}
Karl~Heinrich Hofmann and Klaus Keimel.
\newblock Sheaf theoretical concepts in analysis: {Bundles} and sheaves of
  {Banach} spaces, {Banach} {C}({X})-modules.
\newblock Applications of sheaves, {Proc}. {Res}. {Symp}., {Durham} 1977,
  {Lect}. {Notes} {Math}. 753, 415-441 (1979)., 1979.

\bibitem{hjjm_bdle}
D.~Husem\"{o}ller, M.~Joachim, B.~Jur\v{c}o, and M.~Schottenloher.
\newblock {\em Basic bundle theory and {$K$}-cohomology invariants}, volume 726
  of {\em Lecture Notes in Physics}.
\newblock Springer, Berlin, 2008.
\newblock With contributions by Siegfried Echterhoff, Stefan Fredenhagen and
  Bernhard Kr\"{o}tz.

\bibitem{hus_fib}
Dale Husemoller.
\newblock {\em Fibre bundles}, volume~20 of {\em Graduate Texts in
  Mathematics}.
\newblock Springer-Verlag, New York, third edition, 1994.

\bibitem{iw}
Kenkichi Iwasawa.
\newblock On some types of topological groups.
\newblock {\em Ann. of Math. (2)}, 50:507--558, 1949.

\bibitem{john_approx}
Barry~Edward Johnson.
\newblock Approximately multiplicative maps between {Banach} algebras.
\newblock {\em J. Lond. Math. Soc., II. Ser.}, 37(2):294--316, 1988.

\bibitem{MR650393}
R.~K. Lashof.
\newblock Equivariant bundles.
\newblock {\em Illinois J. Math.}, 26(2):257--271, 1982.

\bibitem{MR885537}
R.~K. Lashof and J.~P. May.
\newblock Generalized equivariant bundles.
\newblock {\em Bull. Soc. Math. Belg. S\'er. A}, 38:265--271, 1986.

\bibitem{MR3331607}
Wolfgang L\"uck and Bernardo Uribe.
\newblock Equivariant principal bundles and their classifying spaces.
\newblock {\em Algebr. Geom. Topol.}, 14(4):1925--1995, 2014.

\bibitem{megr_max-equiv-cpct}
Michael Megrelishvili.
\newblock Maximal equivariant compactifications.
\newblock {\em Topology Appl.}, 329:21, 2023.
\newblock Id/No 108372.

\bibitem{mnk}
James~R. Munkres.
\newblock {\em Topology}.
\newblock Prentice Hall, Inc., Upper Saddle River, NJ, 2000.
\newblock Second edition of [ MR0464128].

\bibitem{pt_brnch}
A.~A. Pavlov and E.~V. Troitskii.
\newblock Quantization of branched coverings.
\newblock {\em Russ. J. Math. Phys.}, 18(3):338--352, 2011.

\bibitem{zbMATH05172034}
N.~Christopher Phillips.
\newblock Recursive subhomogeneous algebras.
\newblock {\em Trans. Am. Math. Soc.}, 359(10):4595--4623, 2007.

\bibitem{schn_cvx_2e_2014}
Rolf Schneider.
\newblock {\em Convex bodies: the {Brunn}-{Minkowski} theory}, volume 151 of
  {\em Encycl. Math. Appl.}
\newblock Cambridge: Cambridge University Press, 2nd expanded ed. edition,
  2014.

\bibitem{ss_countertop}
Lynn~Arthur Steen and J.~Arthur Seebach, Jr.
\newblock {\em Counterexamples in topology}.
\newblock Dover Publications, Inc., Mineola, NY, 1995.
\newblock Reprint of the second (1978) edition.

\bibitem{steen_fib}
Norman Steenrod.
\newblock {\em The topology of fibre bundles}.
\newblock Princeton Landmarks in Mathematics. Princeton University Press,
  Princeton, NJ, 1999.
\newblock Reprint of the 1957 edition, Princeton Paperbacks.

\bibitem{zbMATH03179258}
R.~G. Swan.
\newblock Vector bundles and projective modules.
\newblock {\em Trans. Am. Math. Soc.}, 105:264--277, 1962.

\bibitem{td_transf-gp}
Tammo tom Dieck.
\newblock {\em Transformation groups}, volume~8 of {\em De Gruyter Stud. Math.}
\newblock De Gruyter, Berlin, 1987.

\bibitem{td_alg-top}
Tammo tom Dieck.
\newblock {\em Algebraic topology}.
\newblock EMS Textb. Math. Z{\"u}rich: European Mathematical Society (EMS),
  2008.

\bibitem{MR216463}
Chien Wenjen.
\newblock Concerning paracompact spaces.
\newblock {\em Proc. Japan Acad.}, 43:121--124, 1967.

\bibitem{wil_top}
Stephen Willard.
\newblock {\em General topology}.
\newblock Dover Publications, Inc., Mineola, NY, 2004.
\newblock Reprint of the 1970 original [Addison-Wesley, Reading, MA;
  MR0264581].

\end{thebibliography}
%\bibliographystyle{plain}

% BEGIN INSERTED BBL (bundle-embeddings-extensions-xv1.bbl)
\def\polhk#1{\setbox0=\hbox{#1}{\ooalign{\hidewidth
  \lower1.5ex\hbox{`}\hidewidth\crcr\unhbox0}}}
  \def\polhk#1{\setbox0=\hbox{#1}{\ooalign{\hidewidth
  \lower1.5ex\hbox{`}\hidewidth\crcr\unhbox0}}}
  \def\polhk#1{\setbox0=\hbox{#1}{\ooalign{\hidewidth
  \lower1.5ex\hbox{`}\hidewidth\crcr\unhbox0}}}
  \def\polhk#1{\setbox0=\hbox{#1}{\ooalign{\hidewidth
  \lower1.5ex\hbox{`}\hidewidth\crcr\unhbox0}}}
  \def\polhk#1{\setbox0=\hbox{#1}{\ooalign{\hidewidth
  \lower1.5ex\hbox{`}\hidewidth\crcr\unhbox0}}}
  \def\polhk#1{\setbox0=\hbox{#1}{\ooalign{\hidewidth
  \lower1.5ex\hbox{`}\hidewidth\crcr\unhbox0}}}

% END INSERTED BBL

\Addresses

\end{document}